\documentclass{amsart}

\usepackage{mathtools,amsthm}
\usepackage[letterpaper, margin=1.3in]{geometry}
\usepackage[utf8]{inputenc}
\usepackage[T2A,T1]{fontenc}
\usepackage{hyperref}
\usepackage{enumitem}
\usepackage{mathrsfs}
\usepackage{tikz-cd}
\usepackage{amssymb}
\usepackage{xspace}
\usepackage{nicefrac}
\usepackage{multicol}
\usepackage{etoolbox}
\usepackage[nameinlink]{cleveref}
\usepackage[titletoc]{appendix}
\usepackage{setspace}
\usepackage{bbm}
\usepackage{versions}
\usepackage{stmaryrd}
\usepackage{tensor}
\usepackage{dsfont}

\DeclareFontFamily{U}{wncy}{}
\DeclareFontShape{U}{wncy}{m}{n}{<->wncyr10}{}
\DeclareSymbolFont{mcy}{U}{wncy}{m}{n}
\DeclareMathSymbol{\Sh}{\mathord}{mcy}{"58}
\newcommand{\Sha}{\Sh}

\newtheorem{theorem}{Theorem}[section]
\newtheorem*{theorem*}{Theorem}
\newtheorem{lemma}[theorem]{Lemma}
\newtheorem{corollary}[theorem]{Corollary}
\newtheorem{proposition}[theorem]{Proposition}

\newtheorem*{conjecture*}{Conjecture}
\crefname{theorem}{Theorem}{Theorems}
\crefname{lemma}{Lemma}{Lemmas}
\crefname{corollary}{Corollary}{Corrolaries}
\crefname{proposition}{Proposition}{Propositions}
\crefname{question}{Question}{questions}
\crefname{conjecture}{Conjecture}{Conjectures}

\newtheorem{innercustomgeneric}{\customgenericname}
\providecommand{\customgenericname}{}
\newcommand{\newcustomtheorem}[2]{%
  \newenvironment{#1}[1]
  {%
   \renewcommand\customgenericname{#2}%
   \renewcommand\theinnercustomgeneric{\kern-0.3em ##1}%
   \innercustomgeneric
  }
  {\endinnercustomgeneric}
}

\newcustomtheorem{specialtheorem}{}
\crefname{specialtheorem}{}{}

\theoremstyle{definition}
\newtheorem{definition}[theorem]{Definition}

\newtheorem{assumption}[theorem]{Assumption}
\crefname{definition}{Definition}{Definitions}
\crefname{hypothesis}{Hypothesis}{Hypothesis}
\crefname{assumption}{Assumption}{Assumptions}

\theoremstyle{remark}
\newtheorem{remark}[theorem]{Remark}
\crefname{remark}{Remark}{Remarks}

\crefname{appsec}{Appendix}{Appendices}

\newcommand\xtwoheadrightarrow[2][]{%
  \mathrel{\ooalign{$\xrightarrow[#1\mkern4mu]{#2\mkern4mu}$\cr%
  \hidewidth$\rightarrow\mkern4mu$}}
}
\newcommand{\abs}[1]{\##1}
\newcommand{\Gal}[2]{\mathrm{Gal}\left(#1\slash #2\right)}

\newcommand{\Z}{\mathbb{Z}}
\newcommand{\Q}{\mathbb{Q}}

\newcommand{\F}{\mathbb{F}}
\newcommand{\p}{\mathfrak{p}}

\renewcommand{\P}{\mathfrak{P}}
\renewcommand{\O}{\mathcal{O}}
\renewcommand{\L}{\mathscr{L}}
\newcommand{\T}{\mathbb{T}}
\newcommand{\A}{\mathbb{A}}
\newcommand{\rightiso}{\xrightarrow{\sim}}
\newcommand{\iso}{\simeq}

\newcommand{\conj}[1]{\overline{#1}}
\renewcommand{\index}[2]{\left[#1:#2\right]}
\newcommand{\Cohomology}[4]{\mathrm{H}^{#1}_{{#2}}\left(#3,#4\right)}
\newcommand{\loc}[1]{\mathrm{loc}_{#1}}
\renewcommand{\log}{\mathrm{log}}

\newcommand{\corank}{\mathrm{corank}}
\renewcommand{\div}{\mathrm{div}}
\newcommand{\tor}{\mathrm{tor}}
\renewcommand{\ker}[1]{\mathrm{ker}\left(#1\right)}
\DeclareMathOperator{\Ker}{ker}
\renewcommand{\Im}[1]{\mathrm{im}\left(#1\right)}

\DeclareMathOperator{\Coker}{coker}
\newcommand{\defeq}{\vcentcolon=}

\newcommand{\ac}{\mathrm{ac}}
\newcommand{\Sel}{\mathrm{Sel}}
\newcommand{\BK}{\mathscr{L}_{\mathrm{BK}}}
\newcommand{\Fac}{\mathscr{L}_{\mathrm{ac}}}

\newcommand{\ilim}[1]{\varprojlim\limits_{#1}}

\renewcommand{\f}{\mathrm{f}}

\renewcommand{\t}{\mathrm{t}}
\newcommand{\ur}{\mathrm{ur}}
\newcommand{\Admissible}{\mathrm{Adm}}
\newcommand{\Kolyvagin}{\mathrm{Kol}}
\newcommand{\trchar}{\mathds{1}}
\newcommand{\Fil}{\mathrm{Fil}}
\newcommand{\dcris}[1]{\mathbb{D}_{\mathrm{cris}}(#1)}

\newcommand{\numbersection}{\numberwithin{theorem}{section}\numberwithin{equation}{section}\renewcommand{\theequation}{\thesection.\alph{equation}}}

\let\oldequation\equation
\let\endoldequation\endequation
\newcommand{\tmpArg}{}
\newenvironment{namedequation}{\begin{equation*}}{\end{equation*}}

\renewenvironment{equation}[1][]
{%
    \ifstrempty{#1}{%
        \renewcommand{\tmpArg}{\endnamedequation}%
        \namedequation%
    }{%
        \renewcommand{\tmpArg}{\endoldequation}%
        \oldequation\label{#1}%
    }%
}%
{%
\tmpArg\ignorespacesafterend%
}

\patchcmd{\section}{\normalfont}{\normalfont\LARGE}{}{}

\begin{document}

\numbersection

\title{A proof of Kolyvagin's Conjecture via the BDP main conjecture}
\author{Murilo Zanarella}

\begin{abstract}
We adapt Wei Zhang's proof of Kolyvagin's conjecture for modular abelian varieties over $\mathbb{Q}$ to rely on the BDP main conjecture instead of on the cyclotomic main conjecture. The main ingredient is a reduction to a case that is tractable by the BDP main conjecture, in a similar spirit to Zhang's reduction to the rank one case. By using the BDP main conjecture instead of the cyclotomic main conjecture, our approach is more suitable than Zhang's to extend to modular abelian varieties over totally real fields.
\end{abstract}

\date{September 18, 2019}

\maketitle

\setcounter{tocdepth}{1}
{\onehalfspacing
\tableofcontents}

\section{Introduction}
\subsection{Main result}
A lot of the following notation follows \cite{Wei-Zhang}.

Fix once and for all a prime $p$ with
\begin{equation}[p-big]\tag{$p$-big}
    p\ge5
\end{equation}
and a quadratic imaginary field $K$ of discriminant $D<-4$ such that
\begin{equation}[split]\tag{split}
   p\text{ splits in }K. 
\end{equation}

For a newform $g\in S_2(\Gamma_0(N))$ of weight $2,$ level $N$ and trivial Nebentypus, we denote its field of coefficients by $F=F_g,$ with ring of integers $\O=\O_g.$ Denote by $\p$ a place of $F$ above $p,$ and by $\O_0$ the order of $\O$ generated by the Fourier coefficients of $g.$ Let $\p_0=\p\cap\O_0$ and let
\begin{equation}
    k=\O/\p,\quad k_0=\O_0/\p_0.
\end{equation}
Let $A=A_g$ be its associated $\mathrm{GL}_2$-type abelian variety over $\Q,$ where we choose an isomorphism class with an embedding $\O\hookrightarrow \mathrm{End}_\Q(A).$ Denoting by $\O_\p$ the ring of integers of $F_\p,$ we have a Galois representation
\begin{equation}
    \rho_{A,\p}\colon G_\Q\to\mathrm{GL}_{\O_\p}(T_\p(A))
\end{equation}
on the Tate module $T_\p(A)=\ilim{} A[\p^i],$ which is a free $\O_\p$-module of rank $2.$ As explained in \cite{Carayol}, this representation is defined over the smaller subring $\O_{0,\p_0}\subseteq \O_\p$:
\begin{equation}
    \rho_{A,\p_0}\colon G_\Q\to\mathrm{GL}_2(\O_{0,\p_0})\subseteq \mathrm{GL}_2(\O_\p)
\end{equation}
such that
\begin{equation}[defined-smaller]
    \rho_{A,\p}=\rho_{A,\p_0}\otimes_{\O_{0,\p_0}}\O_\p.
\end{equation}

We consider the reduction of $\rho_{A,\p}$:
\begin{equation}
    \conj{\rho}_{A,\p}\colon G_\Q\to\mathrm{GL}_2(V_k)
\end{equation}
where $V_k=A[\p]$ is a two-dimensional $k$-vector space. Because of \eqref{defined-smaller}, there is a two dimensional $k_0$-vector space $V_g$ such that $V_k=V_g\otimes_{k_0}k$ as Galois modules.

Write $N=N^+N^-$ such that primes $l\mid N^+$ are split or ramified in $K$ and primes $l\mid N^-$ are inert in $K.$ We consider the following assumption on $N.$
\begin{equation}[Heegner]\tag{Heegner}
N^-\text{ is square-free with an even number of prime factors and }\mathrm{gcd}(N,D)=1.
\end{equation}

We also consider the following assumption on the pair $(g,\p)$:
\begin{assumption}\label{ass}
Assume that $g\in S_2(\Gamma_0(N))$ is such that
\begin{equation}[sqf]\tag{$\square$-free}
N \text{ is square-free}
\end{equation}
and
\begin{equation}[good]\tag{good}
p\nmid N.
\end{equation}
We also assume that
\begin{equation}\tag{res-surj}\label{res-surj}
\text{the residual representation }\conj{\rho}_{A,\p_0}\colon G_\Q\to\mathrm{GL}(V_g)\text{ is surjective},
\end{equation}
that
\begin{equation}[ram]\tag{ram}
\mathrm{Ram}(\conj{\rho}_{A,\p})\text{ contains all prime factors }l\parallel N^+\text{ and all }q\mid N^-\text{ such that }q\equiv \pm 1\mod p,
\end{equation}
where $\mathrm{Ram}(\conj{\rho}_{A,\p})$ is the set of places ramified in $\conj{\rho}_{A,\p},$ and that
\begin{equation}[not-anom]\tag{not anom}
    \p\nmid p+1-a_p.
\end{equation}
\end{assumption}

\begin{remark}
    We note that if $g$ satisfies \eqref{res-surj}, then $A$ is defined uniquely up to prime-to-$\p$ isogeny, and hence $\conj{\rho}_{A,\p_0}$ depends only on $g.$ We then denote $\conj{\rho}_{A,\p_0}=\conj{\rho}_{g,\p_0}$. In this case, we may also take $A_g$ to be $(\O,\p)$-optimal in the sense of \cite[Section 3.7]{Wei-Zhang}, and we do so. We also note that \eqref{ram} is equivalent to \cite[Hypothesis $\heartsuit$]{Wei-Zhang} when \eqref{sqf} holds.
\end{remark}
\begin{remark}
We note that \eqref{split}, \eqref{good} and \eqref{not-anom} imply that we have\footnote{Indeed, if $A^1(\Q_p)\subseteq A(\Q_p)$ denotes the kernel of reduction modulo $p,$ we have $0\to A^1(\Q_p)\to A(\Q_p)\to A(\F_p)\to 0$ by \eqref{good}. Applying $\otimes_{\O}\O_\p,$ and using that $A(\F_p)\otimes_\O\O_\p=A[\p^\infty](\F_p)\iso\O_\p/(1-a_p+p)$ is zero by \eqref{not-anom}, we conclude that $A^1(\Q_p)\otimes_\O\O_\p\iso A(\Q_p)\otimes_\O\O_\p.$ But $A^1(\Q_p)\otimes_{\O\otimes\Z_p}\O_\p$ is free of rank $1$ over $\O_\p,$ and this implies that $A[\p^\infty](K_w)=0$ since $K_w\iso\Q_p$ by \eqref{split}.}
\begin{equation}[no-local-tor]\tag{no local tor}
    \Cohomology{0}{}{K_w}{A[\p^\infty]}=0\text{ for }w\mid p.
\end{equation}
\end{remark}

A \emph{Kolyvagin prime} for $g$ is a prime $l\nmid NDp$ that is inert in $K$ and satisfy
\begin{equation}
    p\mid l+1,\text{ and } \p\mid a_l.
\end{equation}
Let $\Kolyvagin_g$ denote the set of square-free products of Kolyvagin primes for $g.$

When $(g,\p)$ satisfy \eqref{Heegner}, \eqref{good} and \eqref{res-surj}, we consider the collection of cohomology classes
\begin{equation}
    \kappa_g=\{c_g(n)\in\Cohomology{1}{}{K}{V_g}\colon n\in\Kolyvagin_g\}
\end{equation}
constructed in \cite[Section 3.7]{Wei-Zhang}, which are the $\mathrm{mod}\ \p$ classes of a Kolyvagin system.

\begin{specialtheorem}{Theorem A}\label{Theorem-A}
    Let $(g,\p)$ be a pair satisfying \cref{ass} and \eqref{Heegner}. Assume that the BDP main conjecture \cite[Conjecture 6.1.2]{JSW} is true for all pairs $(g',\p')$ satisfying \cref{ass}, \eqref{Heegner}, such that $\overline{\rho}_{g,\p_0}\iso\overline{\rho}_{g',\p'_0}$ and such that $N'^->1.$ Then we have $\kappa_g\neq\{0\}.$
\end{specialtheorem}

\begin{remark}
    In the proof of the theorem, we will use the BDP main conjecture for a single pair $(g',\p')$ of level $N'=Nm,$ for a certain product $m>1$ of primes that are inert in $K.$ In fact, we have some choice over $m$: for instance, we may choose it to have an arbitrarily large number of prime factors, while also avoiding any set of primes of density $0.$
\end{remark}

\begin{remark}
Although the current results on the BDP main conjecture do not allow us to obtain new results towards Kolyvagin's conjecture as a corollary\footnote{Although Wei Zhang works in the ordinary case, one may replace the cyclotomic main conjecture of \cite{Skinner-Urban} by the one in \cite{CCSS} to extend Zhang's proof to the supersingular case as well.}, our method is more suitable than Zhang's to extend to modular abelian varieties over totally real fields: the proofs of the cyclotomic main conjectures, for instance in \cite{Skinner-Urban} and \cite{CCSS}, rely on Kato and Beilinson--Flach elements, which we don't have analogues of in the totally real case. On the other hand, the proofs of the BDP main conjecture rely on Heegner points, which are available in the totally real case. We also note that a large part of the methods in \cite{Wei-Zhang} have already been extended to the totally real case in \cite{Haining}.
\end{remark}

\subsection{Proof outline and organization of the paper}
Our proof follows \cite{Wei-Zhang} very closely. There, Zhang performs an induction on the dimension of the $\p$-Selmer group, using the level raising results of \cite{DT2, DT}. He reduces the problem to the cases of dimension $0$ and $1,$ and then uses the results on the cyclotomic main conjecture of \cite{Skinner-Urban} to show that: (i) the dimension $0$ case cannot occur and (ii) the class $c_g(1)$ is nonzero in the dimension $1$ case.

As Zhang already noticed, we can rule out the dimension $0$ case by using the results on the parity conjecture in \cite{Nekovar2}. In the setting we are considering, we may also give a simple proof of the parity conjecture by relying on Howard's formalism of Kolyvagin systems. This is done in \Cref{parity-section}.

The novelty of our paper is how we deal with the dimension $1$ case. We first perform a level raising argument to reduce the problem further to the case where the BDP Selmer group\footnote{The BDP Selmer group is defined when \eqref{split} holds: it has the usual Bloch--Kato local condition for places $w\nmid p,$ the strict condition at one prime above $p,$ and the relaxed condition at the other prime above $p.$} is trivial. Such reduction relies on an extension of the parity lemma of Gross--Parson \cite[Lemma 9]{Gross-Parson}, which we establish in \Cref{cohomology-section}. In the case of dimension $1,$ the logarithm of the Heegner point $P\in A(K)$ can be related to the size of the BDP Selmer group, and the triviality of the latter will imply the $p$-indivisibility of $P.$ Since $c_g(1)$ is the image of $P$ under the Kummer map, the $p$-indivisibility of $P$ amounts to $c_g(1)\neq0.$ Such relation arises from specializing the BDP main conjecture at the trivial character: using the BDP formula of \cite{Brooks} on the analytic side and the anticyclotomic control theorem of \cite{JSW} on the algebraic side. We carry out such argument and conclude the proof of \cref{Theorem-A} in \Cref{proof-section}.

\subsection{A note on the hypothesis}
The hypothesis \eqref{sqf} is only needed in order to use the BDP formula of \cite{Brooks} (see \cite[Proposition 5.1.7]{JSW}). As mentioned in \cite[Section 7.4.4]{JSW}, this can likely be dropped.

The condition \eqref{no-local-tor} seems to be essential to our arguments: it plays an important role in the reduction to the case of trivial BDP Selmer group: it is necessary, for instance, for \eqref{f-is-one-dim} to be true. Moreover, \eqref{not-anom} and \eqref{no-local-tor} are also used to deduce $\kappa_g\neq\{0\}$ from the formula \eqref{bdp-form} obtained from the anticyclotomic control theorem. 

\subsection*{Acknowledgments}
I want to thank Francesc Castella for advising me throughout this project, and for all the encouragement and advice. I am also grateful for him for noting that the parity conjecture in the non-ordinary case could be deduced from the results in \cite{CCSS}. I would also like to thank Daniel Kriz for his willingness to answer many of the questions I had when preparing this paper.

\section{Parity conjecture}\label{parity-section}
In this section, let $(g,\p)$ satisfy \eqref{Heegner}, \eqref{good} and \eqref{res-surj}. Our goal is to prove that
\begin{equation}
    \corank_{\O_\p}\Sel_{\p^\infty}(A/K)\equiv1\mod2,
\end{equation}
where $\Sel_{\p^\infty}(A/K)$ denotes the usual $\p$-adic Selmer group.

As mentioned in the introduction, this is already covered by the result \cite[Theorem B]{Nekovar2}. However, we will give a simple proof of the parity conjecture in our setting by essentially following \cite{Nekovar}: combining a Kolyvagin system argument with an anticyclotomic control theorem.

In the ordinary case, the necessary ingredients are essentially already in \cite{Howard2}. For the non-ordinary case, the control theorm will be a simple consequence of the work of \cite{CCSS} on $\flat/\sharp$-Selmer groups.

For this section, $T=T_\p A$ denotes the $\p$-adic Tate module of $A.$ Let $\Gamma\defeq\Gal{K_\infty^\ac}{K}$ be the anticyclotomic Galois group, with a topological generator $\gamma.$ Let $\Lambda\defeq\O_\p\llbracket\Gamma\rrbracket$ be the Iwasawa algebra with Galois action given by $\Psi\colon G_K\to\Gamma,$ and denote $\Lambda_\infty\defeq\hat{\O}_{F_\infty}\llbracket\Gamma\rrbracket$ where $F_\infty$ is the unramified $\Z_p$-extension of $\Q_p.$ Denote $\T=T\otimes_{\O_\p}\Lambda$ and $\A=T\otimes_{\O_\p}\Lambda^\vee$ with diagonal Galois action, where $G_K$ acts on $\Lambda$ by the natural projection $\Psi\colon G_k\to \Gamma,$ and acts on $\Lambda^\vee$ by $\Psi^{-1}.$

We recall some objects from \cite{CCSS} in the non-ordinary case. We have Coleman maps $\mathrm{Col}^\bullet_v\colon\Cohomology{1}{}{K_v}{\T}\to\Lambda_\infty$ for $\bullet\in\{\flat,\sharp\}$ and $v\mid p,$ as in \cite[Section 4.2]{CCSS}, obtained by restricting the two-variable Coleman maps, first defined in \cite{Buyukboduk-Lei}, to the anticyclotomic line. We denote by $\Cohomology{1}{\bullet}{K_v}{\T}$ the kernel of such map, and define $\Cohomology{1}{\bullet}{K_v}{\A}$ to be its orthogonal complement under local duality. Finally, we denote by $\mathfrak{Sel}^{\bullet,\bullet}(K,\A)$ the Selmer group with the $\bullet$ conditions for $v\mid p,$ and the unramified condition outside of $p.$

\begin{theorem}\label{flat-sharp-control}
    Let $\bullet\in\{\flat,\sharp\}.$ Then the map $\Sel_{\p^\infty}(A/K)\to\mathfrak{Sel}^{\bullet,\bullet}(K,\A)^{\Gamma}$ has finite kernel and cokernel.
\end{theorem}
\begin{proof}
We first analyze the local conditions for $v\mid p.$ Let $z\in\Cohomology{1}{}{K_v}{\T}$ and $z_0\in\Cohomology{1}{}{K_v}{T}$ be its image under the natural map $\Cohomology{1}{}{K_v}{\T}\to\Cohomology{1}{}{K_v}{T}.$ As in \cite[Proposition 2.12]{Hatley-Lei}, we have that $z_0\in\Cohomology{1}{\f}{K_v}{T}$ if and only if $\trchar(\mathrm{Col}^\bullet_{v}(z))=0.$ Since $\Cohomology{1}{\bullet}{K_v}{\T}\subseteq\Cohomology{1}{}{K_v}{\T}$ is defined to be the kernel of $\mathrm{Col}^\bullet_{v},$ this means that we have the following natural map in cohomology
\begin{equation}
    \Cohomology{1}{\bullet}{K_v}{\T}\xrightarrow{a}\Cohomology{1}{\f}{K_v}{T}.
\end{equation}
This also means that the map $c$ in the commutative diagram below has the same kernel as the evaluation at $\trchar$ map.
\begin{equation}
    \begin{tikzcd}
    0\arrow{r}&\Cohomology{1}{\bullet}{K_v}{\T}\arrow{d}{a}\arrow{r}&\Cohomology{1}{}{K_v}{\T}\arrow{d}{b}\arrow{r}{\mathrm{Col}^\bullet_{v}}&\mathrm{Col}^\bullet_{v}\Cohomology{1}{}{K_v}{\T}\arrow{d}{c}\arrow{r}&0\\
    0\arrow{r}&\Cohomology{1}{\f}{K_v}{T}\arrow{r}&\Cohomology{1}{}{K_v}{T}\arrow{r}{\exp^*}&\Fil^0\dcris{V}&
    \end{tikzcd}
\end{equation}
By a Snake lemma, this means that we have an exact sequence
\begin{equation}[eq_control]
    0\to\frac{\mathrm{Col}^\bullet_{v}\Cohomology{1}{}{K_v}{\T}\cap(\gamma-1)\Lambda_{\infty}}{(\gamma-1)\mathrm{Col}^\bullet_{v}\Cohomology{1}{}{K_v}{\T}}\to\Coker a\to\Coker b,
\end{equation}
but as in \cite[Proposition 2.3]{CCSS}, we can prove that
\begin{equation}
    \mathrm{Col}^\bullet_{v}\colon \Cohomology{1}{}{K_v}{\T}\to\Lambda_{\infty}
\end{equation}
has finite cokernel, and hence that the first module in \eqref{eq_control} is finite. We also have that $\Coker b$ is finite, since it is dual to $\ker{\Cohomology{1}{}{K_v}{A[\p^\infty]}\to\Cohomology{1}{}{K_v}{\A}}\iso\A^{G_{K_v}}\slash(\gamma-1)\A^{G_{K_v}},$ which is finite by the proof of \cite[Proposition 3.3.7: case 3(b)]{JSW}. Hence we conclude from \eqref{eq_control} that $\Coker a$ is finite.

Now the control theorem follows from standard arguments in Iwasawa theory: consider the following diagram
\begin{equation}
    \begin{tikzcd}
        0\arrow{r}&\Sel_{\p^\infty}(A/K)\arrow{r}\arrow{d}{s}&\Cohomology{1}{}{K}{A[\p^\infty]}\arrow{r}\arrow{d}{h}&\mathcal{G}_A(K)\arrow{r}\arrow{d}{g}&0\\
        0\arrow{r}&\mathfrak{Sel}^{\bullet,\bullet}(K,\A)^\Gamma\arrow{r}&\Cohomology{1}{}{K}{\A}^\Gamma\arrow{r}&\mathcal{G}_\A(K)^\Gamma&
    \end{tikzcd}
\end{equation}
where
\begin{equation}
    \mathcal{P}_A(K)=\prod_{v}\frac{\Cohomology{1}{}{K_v}{A[\p^\infty]}}{\Cohomology{1}{\f}{K_v}{A[\p^\infty]}}\quad\text{and}\quad\mathcal{P}_\A(K)=\prod_{v\nmid p}\frac{\Cohomology{1}{}{K_v}{\A}}{\Cohomology{1}{\ur}{K}{\A}}\times\prod_{v\mid p}\frac{\Cohomology{1}{}{K_v}{\A}}{\Cohomology{1}{\bullet}{K_v}{\A}}
\end{equation}
and
\begin{equation}
    \mathcal{G}_A(K)=\mathrm{im}\left(\Cohomology{1}{}{K}{A[\p^\infty]}\to\mathcal{P}_A(K)\right)\quad\text{and}\quad\mathcal{G}_\A(K)=\mathrm{im}\left(\Cohomology{1}{}{K}{\A}\to\mathcal{P}_\A(K)\right).
\end{equation}
As in \cite[Lemmas 3.1, 3.2]{Greenberg}, we have that $\Coker h=0$ and that $\Ker h$ is finite.

Let $r_v$ be the factors of the map $r\colon \mathcal{P}_A(K)\to\mathcal{P}_\A(K).$ From the proof of \cite[Proposition 3.3.7]{JSW}, $\Ker r_v$ is finite when $v\nmid p$ and is $0$ when $A\slash K$ has good reduction at $v.$ For $v\mid p,$ the analysis of the local conditions above imply that $\Ker r_v$ is finite, since it is dual to $\Coker a.$ Hence $\Ker r$ finite, and we can conclude so is $\Ker g.$

By a Snake lemma, we conclude that $\Ker s$ and $\Coker s$ are finite.
\end{proof}

\begin{theorem}\label{parity}
Let $(g,\p)$ be a pair that satisfies \eqref{Heegner}, \eqref{good} and \eqref{res-surj}. Then we have
\begin{equation}
    \corank_{\O_\p}\Sel_{\p^\infty}(A/K)\equiv1\mod2.
\end{equation}
\end{theorem}
\begin{proof}
For the ordinary case, we consider the ordinary Selmer group $\mathrm{Sel}^{\mathrm{ord}}(K,\A)$ as in \cite[Definition 3.2.2]{Howard2}. By \cite[Theorem 3.4.2]{Howard2}, we have a pseudo-isomorphism\footnote{As explained in \cite[Theorem 3.1]{BCK}, we may take $M_\P=0$ in \cite[Theorem 3.4.2]{Howard2}.}
\begin{equation}[howard-structure-ord]
    \mathrm{Hom}_{\O_\p}\left(\mathrm{Sel}^{\mathrm{ord}}(K,\A),F_\p/\O_\p\right)\sim\Lambda\oplus M\oplus M
\end{equation}
for a $\Lambda$-torsion module $M.$

Now \eqref{howard-structure-ord} implies that $\mathrm{Sel}^{\mathrm{ord}}(K,\A)^\Gamma$ has odd $\O_\p$-corank. The control theorem \cite[Lemmas 3.2.11, 3.2.12]{Howard2} says that the natural map $\mathrm{Sel}_{\p^\infty}(A/K)\to \mathrm{Sel}^{\mathrm{ord}}(K,\A)^\Gamma$ has finite kernel and cokernel, and from this we may conclude that $\corank_{\O_\p}\mathrm{Sel}_{\p^\infty}(A/K)$ is odd.

For the non-ordinary case, we repeat the argument above, but with $\flat/\sharp$-Selmer groups. For $\bullet\in\{\flat,\sharp\},$ we have, as in the proof of \cite[Theorem 5.7]{CCSS}\footnote{The $\flat,\sharp$-Selmer groups admit Kolyvagin systems in the sense of \cite{Howard2}, as constructed in \cite[Proposition 5.6]{CCSS}, and so \eqref{howard-structure} follows from the proof of \cite[Theorem 3.4.2]{Howard2}.}, that
\begin{equation}[howard-structure]
    \mathrm{Hom}_{\O_\p}\left(\mathfrak{Sel}^{\bullet,\bullet}(K,\A),F_\p/\O_\p\right)\sim\Lambda\oplus M^\bullet\oplus M^\bullet
\end{equation}
for a $\Lambda$-torsion module $M^\bullet.$

Now \eqref{howard-structure} implies that $\mathfrak{Sel}^{\bullet,\bullet}(K,\A)^\Gamma$ has odd $\O_\p$-corank, and together with \cref{flat-sharp-control} this implies that $\corank_{\O_\p}\Sel_{\p^\infty}(A/K)$ is odd.
\end{proof}

\section{Galois cohomology}\label{cohomology-section}
\subsection{Selmer structures}
We recall the setup of \cite[Chapter 2]{Mazur-Rubin} for Selmer structures.

Let $L\slash\Q_p$ be a finite extension and $\O_L$ its ring of integers. Let $F$ be a number field, and $M$ be an $\O_L$-module with a continuous $\O_L$-linear action of $G_F$ that is unramified except for finitely many primes.

\begin{definition}
A Selmer structure $\L=(\L_v)_v$ for $M$ is a collection of $\O$-submodules $\L_v$ indexed by the places of $F$
\begin{equation}
    \L_v\subseteq \Cohomology{1}{}{F_v}{M}
\end{equation}
such that, for all but finitely many $v,$ we have
\begin{equation}
    \L_v=\Cohomology{1}{\mathrm{ur}}{F_v}{M}\defeq \mathrm{Ker}\left(\Cohomology{1}{}{F_v}{M}\to \Cohomology{1}{}{I_v}{M}\right).
\end{equation}

We consider the associated Selmer group
\begin{equation}
    \Cohomology{1}{\L}{F}{M}\defeq \{c\in\Cohomology{1}{}{F}{M}\colon \loc{v}(c)\in\L_v\text{ for all }v\}.
\end{equation}
\end{definition}

We recall a well-known consequence of Poitou--Tate global duality:

\begin{lemma}[{{\cite[Theorem 2.19]{Fermat's-last}}}]
Let $M$ have finite order, and $\L$ be a Selmer structure for $M$. Then
\begin{equation}
    \frac{\abs{\Cohomology{1}{\L}{F}{M}}}{\abs{\Cohomology{1}{\L^*}{F}{M^*}}}=\frac{\abs{\Cohomology{0}{}{F}{M}}}{\abs{\Cohomology{0}{}{F}{M^*}}}\prod_{v}\frac{\abs{\L_v}}{\abs{\Cohomology{0}{}{F_v}{M}}}.
\end{equation}
\end{lemma}

If $M$ is also self-dual, we can rephrase this theorem in a way which will be more useful to us:

\begin{corollary}\label{Poitou-Tate}
Let $M$ have finite order and be self-dual. Let $\L$ be a Selmer structure for $M$. Then
\begin{equation}
    \frac{\abs{\Cohomology{1}{\L}{F}{M}}}{\abs{\Cohomology{1}{\L^*}{F}{M}}}=\prod_{v}\frac{\abs{\L_v}}{\sqrt{\abs{\Cohomology{1}{}{F_v}{M}}}}.
\end{equation}
\end{corollary}

\begin{proof}
By local duality and by the self-duality of $M,$ we have
\begin{equation}[eq_loc-dual]
    \abs{\Cohomology{0}{}{F_v}{M}}=\abs{\Cohomology{2}{}{F_v}{M^*}}=\abs{\Cohomology{2}{}{F_v}{M}}=\abs{\Cohomology{0}{}{F_v}{M^*}}.
\end{equation}

\indent Let 
\begin{equation}
    \Sigma=\left\{v\mid \abs{M}\cdot\infty\right\}\cup\left\{v\colon \L_v\neq\Cohomology{1}{\mathrm{ur}}{F_v}{M}\right\}.
\end{equation}
Then $\Sigma$ is a finite set, and satisfy
\begin{equation}
    v\notin\Sigma\implies \sqrt{\abs{\Cohomology{1}{}{F_v}{M}}}=\abs{\Cohomology{1}{\mathrm{ur}}{F_v}{M}}=\abs{\L_v}.
\end{equation}
So we only need to prove that
\begin{equation}
    \prod_{v\in\Sigma}\frac{\sqrt{\abs{\Cohomology{1}{}{F_v}{M}}}}{\abs{\Cohomology{0}{}{F_v}{M}}}=1.
\end{equation}
The square of the left side of such expression is, by \eqref{eq_loc-dual}, simply
\begin{equation}
    \prod_{v\in\Sigma}\frac{\abs{\Cohomology{1}{}{F_v}{M}}}{\abs{\Cohomology{0}{}{F_v}{M}}\cdot \abs{\Cohomology{2}{}{F_v}{M}}}=\prod_{v\in\Sigma}\frac{1}{\chi(F_v,M)}.
\end{equation}
Using the formulas for the local Euler characteristics in \cite[Theorem 4.45]{Hida} and \cite[Theorem 4.52]{Hida}, we have
\begin{equation}
    \prod_{v\in\Sigma}\chi(F_v,M)=\prod_p\prod_{v\mid p}p^{-e(v)f(v)\nu_p(\abs{M})}\cdot\prod_{v\text{ real}}\abs{M}\cdot\prod_{v\text{ complex}}(\abs{M})^2
\end{equation}
and since $\sum_{v\mid p}e(v)f(v)=\index{F}{\Q}$ and $r+2s=\index{F}{\Q},$ this becomes
\begin{equation}
    \prod_{v\in\Sigma}\chi(F_v,M)=\prod_v\chi(F_v,M)=\left(\prod_pp^{-\nu(\abs{M})}\right)^{\index{F}{\Q}}\cdot (\abs{M})^{\index{F}{\Q}}=1^{\index{F}{\Q}}=1.
    \qedhere
\end{equation}
\end{proof}

\subsection{Local conditions}\label{local-subsection}
We consider the \emph{strict} local condition $\L_v=\{0\}$ and the \emph{relaxed} local condition $\L_v=\Cohomology{1}{}{F_v}{M}.$

\indent Given a Selmer structure $\L$ for $M$ and given products of places $R$ and $S$ that do not share any place, we denote by $\L^R_S$ the Selmer structure that differs by $\L$ by being strict at $S$ and relaxed at $R,$ that is,
\begin{equation}[Selmer-notation]
    (\L^R_S)_v=\left\{
    \begin{array}{ccl}
    \Cohomology{1}{}{F_v}{M} &\text{if}&v\mid R,\\
    0&\text{if}&v\mid S,\\
    \L_v&&\text{otherwise.}
    \end{array}
    \right.
\end{equation}

\indent For the module $V_k=A[\p],$ we also consider the \emph{finite} local condition
\begin{equation}
    \Cohomology{1}{\f}{K_v}{V_k}=\Im{\delta_v},
\end{equation}
where $\delta_v\colon A(K_v)\to \Cohomology{1}{}{K_v}{V_k}$ is the local Kummer map. These form a Selmer structure $\BK.$ We also define the finite condition for $\Cohomology{1}{\f}{K_v}{V}$ by propagation: it is the pre-image of $\Cohomology{1}{\f}{K_v}{V_k}$ in
\begin{equation}
    \Cohomology{1}{}{K_v}{V}\to\Cohomology{1}{}{K_v}{V_k}.
\end{equation}

\indent We note that $V$ is self-dual, and that $\Cohomology{1}{\f}{K_v}{V}$ is its own annihilator under Tate local duality. In particular, we have $\abs{\Cohomology{1}{\f}{K_v}{V}}=\sqrt{\Cohomology{1}{}{K_v}{V}}.$ Moreover, such local conditions are the unramified condition for all but finitely many primes, and hence form a Selmer structure, which we denote by $\tensor*[_g]{\L}{}.$

Since $V$ is self-dual, from \eqref{split} and \eqref{no-local-tor} we have, when $v\mid p,$ that
\begin{equation}
    \abs{\Cohomology{1}{}{K_v}{V}}=\frac{\abs(\Cohomology{0}{}{K_v}{V})^2}{\chi(K_v,V)}=p^{e(v)f(v)\nu_p(\abs{V})}=(\abs k_0)^2,
\end{equation}
and hence
\begin{equation}[f-is-one-dim]
\abs{\Cohomology{1}{\f}{K_v}{V}}=\abs{k_0}.
\end{equation}

We also consider the \emph{transverse} local condition for certain primes. A prime $q\nmid NDp$ is caled \emph{admissible} for $(g,\p)$ if it is inert in $K$ and satisfy
\begin{equation}
    p\nmid q^2-1\quad\text{and}\quad\p\mid(q+1)^2-a_q^2.
\end{equation}
We denote by $\Admissible_g$ the set of square-free products of admissible primes. We also denote by $\Admissible_g(m)\subseteq\Admissible_g$ the subset of elements coprime with $m.$

The following results are from \cite[Lemma 4.2]{Wei-Zhang}. Let $q\in\Admissible_g.$ We have a unique direct sum decomposition
\begin{equation}
    V\iso k_0\oplus k_0(1)
\end{equation}
as $G_{K_q}$-modules, which induces
\begin{equation}
    \Cohomology{1}{}{K_q}{V}=\Cohomology{1}{}{K_q}{k_0}\oplus\Cohomology{1}{}{K_q}{k_0(1)}.
\end{equation}
Moreover, we have the identification $\Cohomology{1}{}{K_q}{k_0}=\Cohomology{1}{\f}{K_q}{V}.$ We define the \emph{transverse} condition to be
\begin{equation}
    \Cohomology{1}{\t}{K_q}{V}\defeq\Cohomology{1}{}{K_q}{k_0(1)},
\end{equation}
and these satisfy
\begin{equation}
    \dim_{k_0}\Cohomology{1}{\f}{K_q}{V}=\dim_{k_0}\Cohomology{1}{\t}{K_q}{V}=1.
\end{equation}

\indent We can extend the notation in \eqref{Selmer-notation} as follows: for a Selmer structure $\L$ for $V,$ we let $\L^R_S(T)$ be the Selmer structure defined by
\begin{equation}
    (\L^R_S(T))_v=\left\{
    \begin{array}{ccl}
    \Cohomology{1}{}{K_v}{V} &\text{if}&v\mid R,\\
    \Cohomology{1}{\t}{K_v}{V}&\text{if}&v\mid T,\\
    0&\text{if}&v\mid S,\\
    \L_v&&\text{otherwise,}
    \end{array}
    \right.
\end{equation}
where $R,S,T$ do not share any places and $T\in\Admissible_g.$

\indent We also record the following well-known application of Chebotarev regarding admissible primes.

\begin{lemma}[{{\cite[Theorem 3.2]{Bertolini-Darmon}}}]\label{enough-admissible}
    Assume \eqref{p-big} and \eqref{res-surj} and let $c\in\Cohomology{1}{}{K}{V}$ be a non trivial class. Then there is a positive density of admissible primes $q$ such that $\loc{q}(c)\neq 0.$
\end{lemma}

\subsection{The parity lemma}
For this section, let $\L=\tensor*[_g]{\L}{}$ with $(g,\p)$ satisfying \cref{ass}.

We record the following consequence of the parity lemma of Gross--Parson \cite[Lemma 9]{Gross-Parson}.

\begin{lemma}[{{\cite[Lemma 5.3]{Wei-Zhang}}}]\label{parity-lemma}
    Let $q\in\Admissible_g$ be a prime. Then we have
    \begin{equation}
        \dim_{k_0}\Cohomology{1}{\L^q}{K}{V}=1+\dim_{k_0}\Cohomology{1}{\L_q}{K}{V}.
    \end{equation}
    Moreover, we have either
    \begin{equation}
        \Cohomology{1}{\L_q}{K}{V}=\Cohomology{1}{\L(q)}{K}{V}
        \quad\text{and}\quad
        \Cohomology{1}{\L^q}{K}{V}=\Cohomology{1}{\L}{K}{V}
    \end{equation}
    or
    \begin{equation}
        \Cohomology{1}{\L_q}{K}{V}=\Cohomology{1}{\L}{K}{V}
        \quad\text{and}\quad
        \Cohomology{1}{\L^q}{K}{V}=\Cohomology{1}{\L(q)}{K}{V}.
    \end{equation}
\end{lemma}

\begin{corollary}\label{rank-lower}
    If $\dim_{k_0}\Cohomology{1}{\L}{K}{V}>0,$ then there is a positive density of admissible primes $q\in\Admissible_g$ such that
    \begin{equation}
        \dim_{k_0}\Cohomology{1}{\L(q)}{K}{V}=\dim_{k_0}\Cohomology{1}{\L}{K}{V}-1.
    \end{equation}
    \end{corollary}
\begin{proof}
    This is a direct application of \cref{parity-lemma} and \cref{enough-admissible}.
\end{proof}

Now we prove an extension of the parity lemma, and deduce as a consequence a Chebotarev-type result that will be useful to deal with the rank one case.

\begin{lemma}\label{relaxed-parity-lemma}
    Let $v$ be any place of $K,$ and let $q\in\Admissible_g$ be a prime. Assume that
    \begin{equation}
        \Cohomology{1}{\L}{K}{V}=\Cohomology{1}{\L_v}{K}{V}\quad\text{and}\quad\Cohomology{1}{\L(q)}{K}{V}=\Cohomology{1}{\L_v(q)}{K}{V},
    \end{equation}
    that is, that both $\Cohomology{1}{\L}{K}{V}$ and $\Cohomology{1}{\L(q)}{K}{V}$ are strict at $v.$
    
    Then an analogous of the parity lemma holds for the relaxed Selmer groups: we have
    \begin{equation}
        \dim_{k_0}\Cohomology{1}{\L^{v,q}}{K}{V}=1+\dim_{k_0}\Cohomology{1}{\L^v_q}{K}{V}.
    \end{equation}
    Moreover, we have either
    \begin{equation}
        \Cohomology{1}{\L^R_q}{K}{V}=\Cohomology{1}{\L^R(q)}{K}{V}
        \text{ and }
        \Cohomology{1}{\L^{R,q}}{K}{V}=\Cohomology{1}{\L^R}{K}{V}, \text{ for both } R=1\text{ and } R =v
    \end{equation}
    or
    \begin{equation}
        \Cohomology{1}{\L^R_q}{K}{V}=\Cohomology{1}{\L^R}{K}{V}
        \text{ and }
        \Cohomology{1}{\L^{R,q}}{K}{V}=\Cohomology{1}{\L^R(q)}{K}{V}, \text{ for both } R=1\text{ and } R =v.
    \end{equation}
\end{lemma}
\begin{proof}
    By \cref{Poitou-Tate}, we have
    \begin{multline*}
        \frac{\abs{\Cohomology{1}{\L^{v,q}}{K}{V}}}{\abs{\Cohomology{1}{\L^v_q}{K}{V}}}=
        \frac{\abs{\Cohomology{1}{\L^{v,q}}{K}{V}}}{\abs{\Cohomology{1}{\L_{v,q}}{K}{V}}}\cdot
        \frac{\abs{\Cohomology{1}{\L_{v,q}}{K}{V}}}{\abs{\Cohomology{1}{\L_v^q}{K}{V}}}\cdot
        \frac{\abs{\Cohomology{1}{\L_v^q}{K}{V}}}{\abs{\Cohomology{1}{\L_q^v}{K}{V}}}
        =\\
        \left(\abs{k_0}\cdot \sqrt{\abs{\Cohomology{1}{}{K_v}{V}}}\right)\cdot \left(
        \frac{\abs{\Cohomology{1}{\L^q_v}{K}{V}}}{\abs{\Cohomology{1}{\L_{v,q}}{K}{V}}}
        \right)^{-1}\cdot\left(\frac{\abs{k_0}}{\sqrt{\abs{\Cohomology{1}{}{K_v}{V}}}}\right),
    \end{multline*}
    which implies that
    \begin{equation}[eq_relating-r]
        \frac{\abs{\Cohomology{1}{\L^{v,q}}{K}{V}}}{\abs{\Cohomology{1}{\L^v_q}{K}{V}}}=
        (\abs{k_0})^2\cdot \left(
        \frac{\abs{\Cohomology{1}{\L^q_v}{K}{V}}}{\abs{\Cohomology{1}{\L_{v,q}}{K}{V}}}
        \right)^{-1}.
    \end{equation}

    \indent \cref{parity-lemma} tell us that $\Cohomology{1}{\L^q}{K}{V}$ and $\Cohomology{1}{\L_q}{K}{V}$ are $\Cohomology{1}{\L}{K}{V}$ and $\Cohomology{1}{\L(q)}{K}{V}$ in some order. By assumption, these are strict at $v,$ that is, we have
    \begin{equation}
        \Cohomology{1}{\L^q}{K}{V}=\Cohomology{1}{\L^q_v}{K}{V}\text{ and }\Cohomology{1}{\L_q}{K}{V}=\Cohomology{1}{\L_{v,q}}{K}{V}.
    \end{equation}
    Hence we have
    \begin{equation}
        \frac{\abs{\Cohomology{1}{\L_v^q}{K}{V}}}{\abs{\Cohomology{1}{\L_{v,q}}{K}{V}}}=\frac{\abs{\Cohomology{1}{\L^q}{K}{V}}}{\abs{\Cohomology{1}{\L_q}{K}{V}}}=\abs{k_0},
    \end{equation}
    and thus \eqref{eq_relating-r} becomes
    \begin{equation}[eq_relating-r2]
        \frac{\abs{\Cohomology{1}{\L^{v,q}}{K}{V}}}{\abs{\Cohomology{1}{\L^v_q}{K}{V}}}=\abs{k_0},
    \end{equation}
    which is the first part of what we want to prove.
    
    Note that \cref{Poitou-Tate} also gives us that
    \begin{equation}
        \frac{\abs{\Cohomology{1}{\L^v}{K}{V}}}{\abs{\Cohomology{1}{\L_v}{K}{V}}}=\frac{\abs{\Cohomology{1}{\L^v(q)}{K}{V}}}{\abs{\Cohomology{1}{\L_v(q)}{K}{V}}},
    \end{equation}
    and hence
    \begin{equation}[eq_same-for-ms]
        \frac{\abs{\Cohomology{1}{\L^v}{K}{V}}}{\abs{\Cohomology{1}{\L^v(q)}{K}{V}}}=
        \frac{\abs{\Cohomology{1}{\L_v}{K}{V}}}{\abs{\Cohomology{1}{\L_v(q)}{K}{V}}}=
        \frac{\abs{\Cohomology{1}{\L}{K}{V}}}{\abs{\Cohomology{1}{\L(q)}{K}{V}}}=(\abs{k_0})^{\pm1}.
    \end{equation}

    Then \eqref{eq_same-for-ms} and \eqref{eq_relating-r2}, together with the inclusions
    \begin{equation}
        \Cohomology{1}{\L^v_q}{K}{V}\subseteq\Cohomology{1}{\L^v}{K}{V},\ \Cohomology{1}{\L^v(q)}{K}{V}\subseteq\Cohomology{1}{\L^{v,q}}{K}{V}
    \end{equation}
    imply the rest of the theorem.
\end{proof}

\begin{corollary}\label{nonzero-localization}
    Let $v$ be any place of $K$ with $\Cohomology{1}{}{K_v}{V}\neq 0,$ and assume that $\Cohomology{1}{\L}{K}{V}=0.$ Then there is a positive density of admissible primes $q$ such that
    \begin{equation*}
        \loc{v}(\Cohomology{1}{\L(q)}{K}{V})\neq 0.
    \end{equation*}
\end{corollary}

\begin{proof}
    We have $\abs{\Cohomology{1}{\L^v}{K}{V}}=\abs{\Cohomology{1}{\f}{K_v}{V}}>0$ by \cref{Poitou-Tate}, so we know that there is a positive density of admissible primes $q$ such that $\loc{q}(\Cohomology{1}{\L^v}{K}{V})\neq 0$ by \cref{enough-admissible}.
    
    We will prove such primes $q$ suffice. Assume by contradiction that we had $\loc{v}(\Cohomology{1}{\L(q)}{K}{V})=0.$ This would mean that 
    \begin{equation*}
        \Cohomology{1}{\L(q)}{K}{V}=\Cohomology{1}{\L_v(q)}{K}{V},
    \end{equation*}
    and since we have $0=\Cohomology{1}{\L}{K}{V}=\Cohomology{1}{\L_q}{K}{V},$ the last statement in \cref{relaxed-parity-lemma} would imply that
    \begin{equation*}
        \Cohomology{1}{\L^v}{K}{V}=\Cohomology{1}{\L^v_q}{K}{V},
    \end{equation*}
    which cannot be true since we chose $q$ such that $\loc{q}(\Cohomology{1}{\L_v}{K}{V})\neq 0.$
\end{proof}

\section{Proof of the main result}\label{proof-section}
\subsection{Level raising}
We denote by $\Admissible_g^+\subseteq\Admissible_g$ the subset of elements with an even number of prime factors.

The following is a summary of some of the properties of the constructions done in \cite{Wei-Zhang}.
\begin{theorem}\label{Wei-results}
    Let $(g,\p)$ be a pair satisfying \cref{ass}. For any $m\in\Admissible_g,$ there is an eigenform $g_m\in S_2(\Gamma_0(Nm))$ and a prime $\p_m$ above $p$ such that $\overline{\rho}_{g,\p_0}\iso\overline{\rho}_{g_m,\p_{m,0}}.$ Furthermore, the pair $(g_m,\p_m)$ satisfy \cref{ass}.
    
    Fix an identification $V_{g_m}\iso V\defeq V_g$ for all $m.$ If $\L=\tensor*[_g]{\L}{},$ then the Selmer structure of $g_m$ is given by $\tensor*[_{g_m}]{\L}{}=\L(m).$
    
    If $(g,\p)$ satisfy \eqref{Heegner} and $m\in\Admissible_g^+,$ then $(g_m,\p_m)$ also satisfy \eqref{Heegner}, and we denote by $c(n,m)\in\Cohomology{1}{\L(m)}{K}{V}$ the class $c_{g_m}(n).$ Then if $mq_1q_2\in\Admissible_g^+$ with $q_1,q_2$ primes, there is a suitable isomorphism 
    \begin{equation}
        \phi^{\f,\t}=\phi^{\f,\t}_{m,q_1,q_2}\colon \Cohomology{1}{\f}{K_{q_1}}{V}\rightiso \Cohomology{1}{\t}{K_{q_2}}{V}
    \end{equation*}
    such that for all $n\in\Kolyvagin_g,$ we have
    \begin{equation*}
        \phi^{\f,\t}(\loc{q_1}(c(n,m)))=\loc{q_2}(c(n,mq_1q_2)).
    \end{equation}
\end{theorem}

\begin{remark}
In the proof of the last statement, Wei Zhang uses a result of Bertolini--Darmon \cite[Theorem 9.3]{Bertolini-Darmon} stated in the ordinary setting, but, as noted already by \cite[Proposition 4.5]{Darmon-Iovita}, the proof of such result also works in the non-ordinary case.
\end{remark}

Wei Zhang then used this $m$-aspect of the classes $c(n,m)$ to reduce the proof of $\kappa_g\neq\{0\}$ to the case of when $\dim_{k_0}\Cohomology{1}{\L}{K}{V}=1,$ as in the following.
\begin{theorem}\label{reduction-rank-1}
    Consider $(g,\p)$ which satisfy both \cref{ass} and \eqref{Heegner}. Then there is $m\in\Admissible_g^+$ with $\dim_{k_0}\Cohomology{1}{\L(m)}{K}{V}=1$ such that $\kappa_{g_m}\neq\{0\}$ implies $\kappa_g\neq\{0\}.$
\end{theorem}
\begin{proof}
    Once we know that $\dim_{k_0}\Cohomology{1}{\L}{K}{V}$ is necessarily odd, this follows as in the proof of \cite[Theorem 9.1]{Wei-Zhang}.
    
    Now note that $\dim_{k_0}\Cohomology{1}{\L}{K}{V}$ being odd follows directly from \cref{parity}: we have
    \begin{equation}
        \dim_{k_0}\Cohomology{1}{\L}{K}{V}=\dim_k\Sel_\p(A/K)
    \end{equation}
    since $\L_v\otimes_{k_0}k=\Cohomology{1}{\f}{K}{V_k}$ by \cite[Theorem 5.2]{Wei-Zhang}, and we have
    \begin{equation}
        \dim_k\Sel_\p(A/K)\equiv\corank_{\O_\p}\Sel_{\p^\infty}(A/K)\mod2
    \end{equation}
    since $\Sel_\p(A/K)=\Sel_{\p^\infty}(A/K)[\p]$ and since the Cassels--Tate pairing on the indivisible quotient of $\Sha(A/K)$ is non-degenerate.
\end{proof}

\subsection{The rank one case}
Now we prove \cref{Theorem-A} in the case when $\dim_{k_0}\Cohomology{1}{\L}{K}{V}=1.$
For this section, fix a prime $v$ of $K$ with $v\mid p,$ so that $p\O_K=v\conj{v}.$ Also fix a pair $(g,\p)$ satisfying \cref{ass}, \eqref{Heegner} and such that $\dim_{k_0}\Cohomology{1}{\L}{K}{V}=1.$

We first show that we can reduce the problem to proving that $c(1,1)\neq0$ when $\Cohomology{1}{\L_v^{\conj{v}}}{K}{V}=0.$

\begin{proposition}\label{suffices-s}
    We have
    \begin{equation}
        \Cohomology{1}{\L_v^{\conj{v}}}{K}{V}=0\iff
        \Cohomology{1}{\L_v}{K}{V}=0.
    \end{equation}
\end{proposition}
\begin{proof}
    As $\Cohomology{1}{\L_v}{K}{V}\subseteq\Cohomology{1}{\L_v^{\conj{v}}}{K}{V},$ the forward implication is clear.
    
    Now assume that $\Cohomology{1}{\L_v}{K}{V}=0$. Then we would also have $\Cohomology{1}{\L_{\conj{v}}}{K}{V}=0,$ since $A$ is defined over $\Q.$
    
    Then \cref{Poitou-Tate}, together with \eqref{f-is-one-dim}, says that $\abs{\Cohomology{1}{\L^{\conj{v}}}{K}{V}}=\abs{\Cohomology{1}{\f}{K_{\conj{v}}}{V}}=\abs{k_0},$ which is also $\abs{\Cohomology{1}{\L}{K}{V}}$ by assumption. Hence $\Cohomology{1}{\L}{K}{V}=\Cohomology{1}{\L^{\conj{v}}}{K}{V}.$ Intersecting the last equality with $\Cohomology{1}{\L_v}{K}{V}$ gives us
    \begin{equation}
        \Cohomology{1}{\L_v}{K}{V}=\Cohomology{1}{\L_v^{\conj{v}}}{K}{V},
    \end{equation}
    and hence $\Cohomology{1}{\L_v^{\conj{v}}}{K}{V}$ is also $0.$
\end{proof}

\begin{lemma}\label{key-lemma}
    Suppose $q_1\in\Admissible_g$ is such that $\Cohomology{1}{\L(q_1)}{K}{V}=0.$ Then there is a positive density of primes $q_2\in\Admissible_g(q_1)$ such that $\Cohomology{1}{\L_v^{\conj{v}}(q_1q_2)}{K}{V}=0.$
\end{lemma}
\begin{proof}
    By \cref{nonzero-localization}, we can choose $q_2$ such that $\loc{v}(\Cohomology{1}{\L(q_1q_2)}{K}{V})\neq0.$ Since \cref{parity-lemma} gives us that $\dim_{k_0}\Cohomology{1}{\L(q_1q_2)}{K}{V}=1,$ this means that $\Cohomology{1}{\L_v(q_1q_2)}{K}{V}=0.$ Together with \cref{suffices-s}, this gives us that $\Cohomology{1}{\L_v^{\conj{v}}(q_1q_2)}{K}{V}=0.$
\end{proof}

\begin{theorem}\label{reduction-bdp}
    There exist two primes $q_1,q_2\in\Admissible_g$ such that both $\Cohomology{1}{\L_v^{\conj{v}}(q_1q_2)}{K}{V}=0$ and $\dim_{k_0}\Cohomology{1}{\L(q_1q_2)}{K}{V}=1.$ Moreover, if $c(1,q_1q_2)\neq0,$ then also $c(1,1)\neq0.$
\end{theorem}
\begin{proof}
    By \cref{parity-lemma} and \cref{nonzero-localization}, we can choose $q_1$ such that $\Cohomology{1}{\L(q_1)}{K}{V}=0.$ Then, by \cref{key-lemma}, we can choose $q_2$ with $\Cohomology{1}{\L_v^{\conj{v}}(q_1q_2)}{K}{V}=0.$ Note that $\dim_{k_0}\Cohomology{1}{\L(q_1q_2)}{K}{V}=1$ is automatic by \cref{parity-lemma}.
    
    Since $\Cohomology{1}{\L(q_1)}{K}{V}=0$ but $\Cohomology{1}{\L(q_1q_2)}{K}{V}\neq0,$ by \cref{parity-lemma} we must have
    \begin{equation}
        0=\Cohomology{1}{\L(q_1)}{K}{V}=\Cohomology{1}{\L_{q_2}(q_1)}{K}{V}
    \end{equation}
    and hence we have an injection
    \begin{equation}
        \loc{q_2}\colon\Cohomology{1}{\L(q_1q_2)}{K}{V}\hookrightarrow\Cohomology{1}{}{K_{q_2}}{V}.
    \end{equation}
    
    Now if we have $c(1,q_1q_2)\neq0,$ this means that $\loc{q_2}(c(1,q_1q_2))\neq0.$ Since we have
    \begin{equation}
        \loc{q_2}(c(1,q_1q_2))=\phi^{\f,\t}(\loc{q_1}(c(1,1)))
    \end{equation}
    by \cref{Wei-results}, where $\phi^{\f,\t}$ is an isomorphism, we conclude that $c(1,1)\neq0.$
\end{proof}

Now \cref{Theorem-A} is reduced to proving the following.

\begin{theorem}\label{proof-bdp}
    Assume in addition that $\Cohomology{1}{\L_v^{\conj{v}}}{K}{V}=0.$ Assume that the BDP main conjecture \cite[Conjecture 6.1.2]{JSW} holds true for $g.$ Then $c(1,1)\neq0.$
\end{theorem}
\begin{proof}
    Let $W=A[\p^\infty].$ We denote by $\Cohomology{1}{\f}{K_w}{W}$ the usual Bloch--Kato local conditions. They form a Selmer structure $\BK$ such that $\Cohomology{1}{\BK}{K}{W}$ is the usual $\p$-adic Selmer group of $A.$ We also denote by $\Fac$ the Selmer structure such that
    \begin{equation}
        \Cohomology{1}{\Fac}{K_w}{W}=\left\{\begin{array}{ll} \Cohomology{1}{}{K_{\conj{v}}}{W}_{\mathrm{div}}&\text{if }w=\conj{v},\\
        0&\text{otherwise},
        \end{array}\right.
    \end{equation}
    as in \cite[Section 2.2.2]{JSW}.
    
    We will use the results of \cite[Section 3]{JSW}, and so we will need to prove the following two hypothesis:
    \begin{equation}[corank-1]\tag{corank 1}
        \Cohomology{1}{\BK}{K}{W}_\div\iso F_\p/\O_\p\text{ and }\Cohomology{1}{\f}{K_w}{W}\iso F_\p/\O_\p, \text{ for }w\mid p,
    \end{equation}
    \begin{equation}[sur]\tag{sur}
        \Cohomology{1}{\BK}{K}{W}\xtwoheadrightarrow{\loc{w}}\Cohomology{1}{\f}{K_w}{W},\text{ for } w\mid p.
    \end{equation}
    
    We already noted in the proof of \cref{reduction-rank-1} that $\corank_{\O_\p}\Cohomology{1}{\BK}{K}{W}$ is odd, and since $\Cohomology{1}{\BK}{K}{W}[\p]\iso\Cohomology{1}{\BK}{K}{V_k}\iso k,$ we must have $\Cohomology{1}{\BK}{K}{W}_\div\iso F_\p/\O_\p.$
    
    For $w\mid p,$ we consider the diagram
    \begin{equation}
    \begin{tikzcd}
        \Cohomology{1}{\BK}{K}{V_k}\arrow{d}{\loc{w}}\arrow{r}{\sim}&\Cohomology{1}{\BK}{K}{W}_\div[\p]\arrow{d}{\loc{w}}\\
        \Cohomology{1}{\f}{K_w}{V_k}\arrow{r}{\sim}&\Cohomology{1}{\f}{K_w}{W}[\p]
    \end{tikzcd}
    \end{equation}
    Note that the leftmost vertical map is an isomorphism: it is injective as $\Cohomology{1}{\L_v^{\conj{c}}}{K}{V}=0,$ and both domain and codomain have dimension $1$ over $k$ by \eqref{f-is-one-dim}. This implies that the rightmost vertical map is also an isomorphism. Since $\Cohomology{1}{\f}{K_w}{W}$ is divisible, this implies both \eqref{corank-1} and \eqref{sur}.
    
    Let $P\in A(K)$ be the Heegner point defined in \cite[Equation 3.22]{Wei-Zhang}. Since we are assuming the BDP main conjecture for $g$ and we proved \eqref{corank-1} and \eqref{sur}, we can use \cite[Proposition 6.2.1]{JSW} together with the BDP formula of \cite{Brooks} as in \cite[Proposition 5.1.7]{JSW} to obtain
    \begin{equation}[bdp-form]
        2\cdot\nu_\p\left(\frac{1+p-a_p}{p}\cdot\log_{\omega_A}(P)\right)\le\nu_\p\left( \abs{\Cohomology{1}{\Fac}{K}{W}}\cdot C(W)\right),
    \end{equation}
    where $\log_{\omega_A}\colon A(K_v)_{/\tor}\otimes_{\Z_p}\O_\p\to\O_\p$ is $\O_\p$-linear and
    \begin{equation}
        C(W)=\abs{\Cohomology{0}{}{K_v}{W}}\abs{\Cohomology{0}{}{K_{\conj{v}}}{W}}\prod_{w\mid N^+}\abs{\Cohomology{1}{\mathrm{ur}}{K_w}{W}}.
    \end{equation}
    By \eqref{no-local-tor} and \eqref{ram}, we have $C(W)=1.$
    
    Now note that under $\Cohomology{1}{}{K}{A[\p]}\rightiso\Cohomology{1}{}{K}{W}[\p],$ the pre-image of $\Cohomology{1}{\Fac}{K}{W}[\p]$ is contained in $\Cohomology{1}{\L^{\conj{v}}_v}{K}{W}$: indeed, for any place $w$ of $K$ we have the diagram
    \begin{equation}
        \begin{tikzcd}
        0\arrow{r}&\Cohomology{1}{}{K}{A[\p]}\arrow{r}\arrow{d}{\loc{w}}&\Cohomology{1}{}{K}{W}[\p]\arrow{r}\arrow{d}{\loc{w}}&0\\
        A(K_w)[\p^\infty]\arrow{r}{\delta_w}&\Cohomology{1}{}{K_w}{A[\p]}\arrow{r}&\Cohomology{1}{}{K_w}{W}[\p]\arrow{r}&0
        \end{tikzcd}
    \end{equation}
    and the compatibility of the local conditions is clear from the diagram except for $w=v,$ but in this case \eqref{no-local-tor} implies that $A(K_v)[\p^\infty]=0.$ 
    
    Since we are assuming $\Cohomology{1}{\L^{\conj{v}}_v}{K}{W}=0,$ the above implies that $\Cohomology{1}{\Fac}{K}{W}=0.$ Together with \eqref{not-anom}, \eqref{bdp-form} then becomes
    \begin{equation}
        \nu_\p\left(\frac{\log_{\omega_A}(P)}{p}\right)\le0.
    \end{equation}
    But in fact, \eqref{not-anom} and \eqref{good} imply that $\log_{\omega_A}$ take values in $p\O_\p$ (see \cite[Section 3.5]{JSW}), and so the above equation implies that $P$ is not $p$-divisible in $A(K_v)_{/\tor}\otimes_{\Z_p}\O_\p.$ Since $c(1,1)$ is the image of $P$ under the Kummer map, this let us conclude that $c(1,1)\neq0.$
\end{proof}

\begin{proof}[Proof of \cref{Theorem-A}]
    Let $m\in\Admissible_g^+$ be as in \cref{reduction-rank-1}. Now choose $q_1,q_2$ as in \cref{reduction-bdp} for $g_m.$ Note that $g_{mq_1q_2}$ has level $Nmq_1q_2$ and that $(Nmq_1q_2)^-=N^-mq_1q_2>1,$ and in particular satisfies \eqref{Heegner}.
    
    This means that the BDP main conjecture for $g_{mq_1q_2}$ is true by our assumptions. So by \cref{proof-bdp} applied to $g_{mq_1q_2},$ we conclude that $c(1,mq_1q_2)\neq0.$ This suffices to conclude that $c(1,m)\neq0$ by the last part of \cref{reduction-bdp}, and hence that $\kappa_g\neq\{0\}$ by the last part of \cref{reduction-rank-1}.
\end{proof}

\newpage
\begingroup
\setstretch{2}
\bibliographystyle{alpha}
\bibliography{references}

\begin{thebibliography}{C{\relax\c{C}}SS18}

\bibitem[BCK19]{BCK}
Ashay Burungale, Francesc Castella, and Chan-Ho Kim.
\newblock A proof of {P}errin-{R}iou's {H}eegner point main conjecture, 2019.
\newblock arXiv:1908.09512.

\bibitem[BD05]{Bertolini-Darmon}
M.~Bertolini and H.~Darmon.
\newblock Iwasawa's main conjecture for elliptic curves over anticyclotomic
  {$\mathbb{Z}_p$}-extensions.
\newblock {\em Ann. of Math. (2)}, 162(1):1--64, 2005.

\bibitem[BL19]{Buyukboduk-Lei}
Kâzım Büyükboduk and Antonio Lei.
\newblock {Iwasawa Theory of Elliptic Modular Forms Over Imaginary Quadratic
  Fields at Non-ordinary Primes}.
\newblock {\em International Mathematics Research Notices}, 07 2019.

\bibitem[Car94]{Carayol}
Henri Carayol.
\newblock Formes modulaires et repr\'esentations galoisiennes \`a valeurs dans
  un anneau local complet.
\newblock In {\em {$p$}-adic monodromy and the {B}irch and {S}winnerton-{D}yer
  conjecture ({B}oston, {MA}, 1991)}, volume 165 of {\em Contemp. Math.}, pages
  213--237. Amer. Math. Soc., Providence, RI, 1994.

\bibitem[C{\relax\c{C}}SS18]{CCSS}
Francesc Castella, Mirela {\relax\c{C}}iperiani, Christopher Skinner, and
  Florian Sprung.
\newblock On the {I}wasawa main conjectures for modular forms at non-ordinary
  primes, April 2018.
\newblock arxiv:1804.10993.

\bibitem[DDT94]{Fermat's-last}
Henri Darmon, Fred Diamond, and Richard Taylor.
\newblock Fermat's last theorem.
\newblock In {\em Current developments in mathematics, 1995 ({C}ambridge,
  {MA})}, pages 1--154. Int. Press, Cambridge, MA, 1994.

\bibitem[DI08]{Darmon-Iovita}
Henri Darmon and Adrian Iovita.
\newblock The anticyclotomic main conjecture for elliptic curves at
  supersingular primes.
\newblock {\em J. Inst. Math. Jussieu}, 7(2):291--325, 2008.

\bibitem[DT94a]{DT2}
Fred Diamond and Richard Taylor.
\newblock Lifting modular mod {$l$} representations.
\newblock {\em Duke Math. J.}, 74(2):253--269, 1994.

\bibitem[DT94b]{DT}
Fred Diamond and Richard Taylor.
\newblock Nonoptimal levels of mod {$l$} modular representations.
\newblock {\em Invent. Math.}, 115(3):435--462, 1994.

\bibitem[GP12]{Gross-Parson}
Benedict~H. Gross and James~A. Parson.
\newblock On the local divisibility of {H}eegner points.
\newblock In {\em Number theory, analysis and geometry}, pages 215--241.
  Springer, New York, 2012.

\bibitem[Gre99]{Greenberg}
Ralph Greenberg.
\newblock Iwasawa theory for elliptic curves.
\newblock In {\em Arithmetic theory of elliptic curves ({C}etraro, 1997)},
  volume 1716 of {\em Lecture Notes in Math.}, pages 51--144. Springer, Berlin,
  1999.

\bibitem[HB15]{Brooks}
Ernest Hunter~Brooks.
\newblock Shimura curves and special values of {$p$}-adic {$L$}-functions.
\newblock {\em Int. Math. Res. Not. IMRN}, (12):4177--4241, 2015.

\bibitem[Hid00]{Hida}
Haruzo Hida.
\newblock {\em Modular forms and {G}alois cohomology}, volume~69 of {\em
  Cambridge Studies in Advanced Mathematics}.
\newblock Cambridge University Press, Cambridge, 2000.

\bibitem[HL19]{Hatley-Lei}
Jeffrey Hatley and Antonio Lei.
\newblock Arithmetic properties of signed selmer groups at non-ordinary primes.
\newblock {\em Annales de l'Institut Fourier}, 69(3):1259--1294, 2019.

\bibitem[How04]{Howard2}
Benjamin Howard.
\newblock Iwasawa theory of {H}eegner points on abelian varieties of {$\rm
  GL_2$} type.
\newblock {\em Duke Math. J.}, 124(1):1--45, 2004.

\bibitem[JSW17]{JSW}
Dimitar Jetchev, Christopher Skinner, and Xin Wan.
\newblock The {B}irch and {S}winnerton-{D}yer formula for elliptic curves of
  analytic rank one.
\newblock {\em Camb. J. Math.}, 5(3):369--434, 2017.

\bibitem[MR04]{Mazur-Rubin}
Barry Mazur and Karl Rubin.
\newblock Kolyvagin systems.
\newblock {\em Mem. Amer. Math. Soc.}, 168(799):viii+96, 2004.

\bibitem[Nek01]{Nekovar}
Jan Nekov\'{a}\v{r}.
\newblock On the parity of ranks of {S}elmer groups. {II}.
\newblock {\em C. R. Acad. Sci. Paris S\'{e}r. I Math.}, 332(2):99--104, 2001.

\bibitem[Nek13]{Nekovar2}
Jan Nekov\'{a}\v{r}.
\newblock Some consequences of a formula of {M}azur and {R}ubin for arithmetic
  local constants.
\newblock {\em Algebra Number Theory}, 7(5):1101--1120, 2013.

\bibitem[SU14]{Skinner-Urban}
Christopher Skinner and Eric Urban.
\newblock The {I}wasawa main conjectures for {$\rm GL_2$}.
\newblock {\em Invent. Math.}, 195(1):1--277, 2014.

\bibitem[Wan15]{Haining}
Haining Wang.
\newblock {\em Anticyclotomic {I}wasawa theory for {H}ilbert modular forms}.
\newblock ProQuest LLC, Ann Arbor, MI, 2015.
\newblock Thesis (Ph.D.)--The Pennsylvania State University.

\bibitem[Zha14]{Wei-Zhang}
Wei Zhang.
\newblock Selmer groups and the indivisibility of {H}eegner points.
\newblock {\em Camb. J. Math.}, 2(2):191--253, 2014.

\end{thebibliography}
\endgroup

\end{document}